\documentclass[12pt]{amsart}
\usepackage{bbold,hyperref,cleveref}

\theoremstyle{plain}
\newtheorem{theorem}{Theorem}[section]
\newtheorem{proposition}[theorem]{Proposition}
\newtheorem{corollary}[theorem]{Corollary}
\newtheorem{lemma}[theorem]{Lemma}

\theoremstyle{definition}
\newtheorem{remark}[theorem]{Remark}
\newtheorem{question}[theorem]{Question}

\newtheorem{example}[theorem]{Example}

\newcommand{\abs}[1]{\lvert#1\rvert}
\newcommand{\norm}[1]{\lVert#1\rVert}
\newcommand{\bigabs}[1]{\bigl\lvert#1\bigr\rvert}
\newcommand{\bignorm}[1]{\bigl\lVert#1\bigr\rVert}

\newcommand{\Bignorm}[1]{\Bigl\lVert#1\Bigr\rVert}
\newcommand{\term}[1]{{\textit{\textbf{#1}}}}   
\renewcommand{\mid}{\::\:}

\newcommand{\goeso}{\xrightarrow{\mathrm{o}}}	
\newcommand{\goesun}{\xrightarrow{\mathrm{un}}} 
\newcommand{\goesuo}{\xrightarrow{\mathrm{uo}}}	
\newcommand{\goesmu}{\xrightarrow{\mu}}	
\newcommand{\goesw}{\xrightarrow{\mathrm{w}}}	
\newcommand{\goesae}{\xrightarrow{\mathrm{a.e.}}}	

\def\one{\mathbb 1}

\DeclareMathOperator{\Range}{Range}

\begin{document}

\title[Unbounded norm convergence]
{Unbounded Norm Convergence\\ in Banach Lattices}

\author{Y. Deng}
\address{School of Mathematics, Southwest Jiaotong University,
  Chengdu, Sichuan, 610000, China.}
\email{15208369715@163.com}

\author{M. O'Brien}
\address{Department of Mathematical and Statistical Sciences,
         University of Alberta, Edmonton, AB, T6G\,2G1, Canada.}
\email{mjbeauli@ualberta.ca}

\author{V.G. Troitsky}
\address{Department of Mathematical and Statistical Sciences,
         University of Alberta, Edmonton, AB, T6G\,2G1, Canada.}
\email{troitsky@ualberta.ca}

\thanks{The third author was supported by an NSERC grant.}
\keywords{Banach lattice, un-convergence, uo-convergence, order convergence, AL-representation}
\subjclass[2010]{Primary: 46B42. Secondary: 46A40}

\date{\today}

\begin{abstract}
   A net $(x_\alpha)$ in a vector lattice $X$ is unbounded order convergent to $x \in X$ if $\abs{x_\alpha - x} \wedge u$ converges to $0$ in order for all $u\in X_+$. This convergence has been investigated and applied in several recent papers by Gao et al. It may be viewed as a generalization of almost everywhere convergence to general vector lattices.
In this paper, we study a variation of this convergence for Banach lattices.  
A net $(x_\alpha)$ in a Banach lattice $X$ is unbounded norm convergent to $x$ if $\bignorm{\abs{x_\alpha - x} \wedge u}\to 0$ for all $u\in X_+$. We show that this convergence may be viewed as a generalization of convergence in measure. We also investigate its relationship with other convergences.
\end{abstract}

\maketitle

\section{Introduction}
We begin by recalling a few definitions. 
A net $(x_\alpha)_{\alpha \in A}$ in a vector lattice $X$ is said to be \term{order convergent} to $x\in X$ if there is a net $(z_\beta)_{\beta \in B}$ in $X$ such that $z_\beta \downarrow 0$ and for every $\beta \in B$, there exists $\alpha_0 \in A$ such that $\abs{x_\alpha -x} \leq z_\beta$ whenever $\alpha \geq \alpha_0$. For short, we will denote this convergence by $x_\alpha \goeso x$ and write that $x_\alpha$ is o-convergent to $x$. A net $(x_\alpha)_{\alpha \in A}$ in a vector lattice $X$ is \term{unbounded order convergent} to $x \in X$ if $\abs{x_\alpha - x} \wedge u \goeso 0$ for all $u\in X_+$. We denote this convergence by $x_\alpha \goesuo x$ and write that $x_\alpha$ uo-converges to $x$. We refer the reader to \cite{GTX} for a detailed exposition on uo-convergence and further references. In particular, for sequences in $L_p(\mu)$ spaces, uo-convergence agrees with almost everywhere (a.e.)\ convergence. Furthermore, if $X$ can be represented as an ideal (or, more generally, a regular sublattice) in $L_1(\mu)$, then the uo-convergence of sequences in $X$ agrees with the a.e.\ convergence in $L_1(\mu)$. Thus, uo-convergence may be viewed as a generalization of a.e.\ convergence to general vector lattices.

Throughout this paper, $X$ will stand for a Banach lattice. For a net $(x_\alpha)$ in $X$ we write $x_\alpha\to x$ if $(x_\alpha)$ converges to $x$ in norm, i.e., $\norm{x_\alpha-x}\to 0$. A net $(x_\alpha)$ in $X$ is \term{unbounded norm convergent} or \term{un-convergent} to $x$ if $\abs{x_\alpha - x} \wedge u\to 0$ for all $u\in X_+$; we then write $x_\alpha \goesun x$. Un-convergence was introduced in \cite{Troitsky:04} as a tool to study measures of non-compactness. In this paper, we study properties of un-convergence and its relationship to other convergences. In particular, we show that un-convergence may be viewed as a generalization of convergence in measure to general Banach lattices.

\section{Basic properties of un-convergence} 

Unless stated otherwise, we will assume that $X$ is a Banach lattice and all nets and vectors lie in $X$. We routinely use the following inequality:
\begin{math}
  (x+y)\wedge u\le x\wedge u+y\wedge u
\end{math}
for all $x,y,u\in X^+$.

\begin{lemma} \label{basic}
\begin{enumerate}
\item\label{basic:diff} $x_\alpha \goesun x$ iff $(x_\alpha -x) \goesun 0$;
\item\label{basic:subs} If $x_\alpha \goesun x$, then $y_\beta \goesun x$ for any subnet $(y_\beta)$ of $(x_\alpha)$.
\item\label{basic:alg} Suppose $x_\alpha \goesun x$ and $y_\alpha \goesun y$. Then $ax_\alpha +by_\alpha \goesun ax+by$ for any $a, b \in \mathbb{R}$.
\item\label{basic:unique} If $x_\alpha \goesun x$ and $x_\alpha\goesun y$, then $x=y$.
\item\label{basic:mod} If $x_\alpha\goesun x$, then $\abs{x_\alpha} \goesun \abs{x}$.
\end{enumerate} 
\end{lemma}

\begin{proof}
\eqref{basic:diff}, \eqref{basic:subs}, and \eqref{basic:alg} are straightforward. To prove~\eqref{basic:unique}, observe that
\begin{math}
  \abs{x-y}\le\abs{x-x_\alpha}+\abs{y-x_\alpha}
\end{math}
for every $\alpha$. Put $u=\abs{x-y}$; it follows that
\begin{displaymath}
  \abs{x-y}=\abs{x-y}\wedge u
  \le\abs{x-x_\alpha}\wedge u+\abs{y-x_\alpha}\wedge u\to 0.
\end{displaymath}
Finally, \eqref{basic:mod} follows from $\bigabs{\abs{x_\alpha}-\abs{x}}\le\abs{x_\alpha-x}$.
\end{proof}

\begin{remark}
  In particular, $x_\alpha\goesun x$ iff $\abs{x_\alpha-x}\goesun 0$. This often allows one to reduce general un-convergence to un-convergence of positive nets to zero.   
\end{remark}

\begin{example}
  It was observed in Examples~21 and~22 in \cite{Troitsky:04} that on $c_0$ un-convergence agrees with coordinate-wise convergence and on $C_0(\Omega)$ un-convergence agrees with uniform convergence on compacta. 
\end{example}

The proofs of the following two facts are straightforward.

\begin{proposition}\label{norm-un}
  If $x_\alpha\to 0$ then $x_\alpha\goesun 0$. For order bounded nets, un-convergence and norm convergence agree.
\end{proposition}

This justifies  the name \textit{unbounded} norm convergence.

\begin{proposition}\label{uo-oc-un}
  In an order continuous Banach lattice, uo-convergence implies un-convergence. 
\end{proposition}

\begin{example}\label{disj-ellinfty}
  Let $(e_n)$ be the standard unit sequence $(e_n)$ in $\ell_\infty$. It follows immediately from \Cref{norm-un} that it is not un-null. This serves as a counterexample to several ``natural'' statements. First, while \cite[Corollary~3.6]{GTX} asserts that every disjoint sequence is uo-null (see also \cite[Lemma~1.1]{Gao:14}), this example shows that a disjoint sequence need not be un-null. Second, since $(e_n)$ is uo-null, this example shows that the order continuity assumption in \Cref{uo-oc-un} cannot be dropped.

Third, it was observed in \cite[Theorem~3.2]{GTX} that for a net $(x_\alpha)$ in a regular sublattice $F$ of a vector lattice $E$, $x_\alpha\goesuo 0$ in $F$ iff $x_\alpha\goesuo 0$ in $E$. This fails for un-convergence: indeed, $(e_n)$ is un-null as a sequence in $c_0$, but not in $\ell_\infty$. However, it is easy to see that  if $x_\alpha\goesun 0$ in $X$ then  $x_\alpha\goesun 0$ in every sublattice of $X$.
\end{example}

\begin{example}\label{un-sublat-not}
  The next example shows that un-convergence in a sublattice does not imply un-convergence in the entire space even when the sublattice is a lattice copy of $\ell_1$. Let $X=\ell_1\oplus_\infty\ell_\infty$; let $(f_n)$ be the standard unit basis of $\ell_1$ and $(g_n)$ the standard unit sequence in $\ell_\infty$. Put $x_n=f_n\oplus g_n$. Let $Y$ be the closed span of $(x_n)$ in $X$. Since $(x_n)$ is a disjoint sequence in $X$, $Y$ is exactly the closed sublattice generated by $(x_n)$. Observe that
  \begin{displaymath}
    \Bignorm{\sum_{k=1}^n\alpha_kx_k}
    =\Bignorm{\sum_{k=1}^n\alpha_kf_k}\vee\Bignorm{\sum_{k=1}^n\alpha_kg_k}
    =\Bigl(\sum_{k=1}^n\abs{\alpha_k}\Bigr)\vee\Bigl(\bigvee_{k=1}^n\abs{\alpha_k}\Bigr)
    =\sum_{k=1}^n\abs{\alpha_k}.
  \end{displaymath}
for any $n$ and any scalars $\alpha_1,\dots,\alpha_n$. It follows that the basic sequence $(x_n)$ in $X$ is 1-equivalent to $(f_n)$ in $\ell_1$ and, therefore, 
$Y$ is an isometric lattice copy of $\ell_1$ in $X$. It is easy to see that $f_n\goesun 0$ in $\ell_1$; hence $x_n\goesun 0$ in $Y$.

However, we claim that $x_n\not\goesun 0$ in $X$. Indeed, let
$u=0\oplus\one=\bigvee_{k=1}^\infty g_k$. Then $x_n\wedge u=g_n$ for
every $n$, hence $(x_n\wedge u)$ does not converge to zero in $X$.
\end{example}

The following three results are similar to Lemmas~3.6 and~3.7, and Proposition~3.9 of \cite{GaoX:14}; we replace uo-convergence with un-convergence, and we do not require the space be order continuous in this case. The proofs are similar.

\begin{lemma}
  If $x_\alpha\goesun x$ then $\abs{x_\alpha}\wedge\abs{x}\to\abs{x}$ and $\norm{x}\le\liminf_\alpha\norm{x_\alpha}$.
\end{lemma}

Recall that a subset $A$ of $X$ is \term{almost order bounded} if for every $\varepsilon>0$ there exists $u\in X_+$ such that $A\subseteq [-u,u]+\varepsilon B_X$. Equivalently, $\bignorm{\bigl(\abs{x}-u)^+}<\varepsilon$ for all $x\in A$.

\begin{lemma}
  If  $x_\alpha\goesun x$ and $(x_\alpha)$ is almost order bounded then $x_\alpha\to x$.
\end{lemma}

\begin{proposition}
  If $(x_\alpha)$ is relatively weakly compact and $x_\alpha\goesun x$ then $(x_\alpha)$ converges to $x$ in $\abs{\sigma}(X,X^*)$.
\end{proposition}

\begin{proof}
  Without loss of generality, $x=0$. Let $f\in X^*_+$. Fix $\varepsilon>0$. By \cite[Theorem 4.37]{Aliprantis:06}, there exists $u\in X_+$ such that
  \begin{math}
    f\Bigl(\bigl(\abs{x_\alpha}-u\bigr)^+\Bigr)<\varepsilon
  \end{math}
  for every $\alpha$. It follows from $\abs{x_\alpha}\wedge u\to 0$ that
  \begin{math}
    f\bigl(\abs{x_\alpha}\wedge u\bigr)\to 0,
  \end{math}
  so that
  \begin{displaymath}
    f\bigl(\abs{x_\alpha}\bigr)=
    f\bigl(\abs{x_\alpha}\wedge u\bigr)
     +f\Bigl(\bigl(\abs{x_\alpha}-u\bigr)^+\Bigr)
    <2\varepsilon
  \end{displaymath}
  for all sufficiently large $\alpha$. It follows that
  \begin{math}
    f\bigl(\abs{x_\alpha}\bigr)\to 0.
  \end{math}
\end{proof}

If $(x_\alpha)$ is a net in a vector lattice with a weak unit $e$ then
$x_\alpha\goesuo 0$ iff $\abs{x_\alpha}\wedge e\goeso 0$; see, e.g.,
\cite[Lemma~3.5]{GTX}. Analogously, the next result limits the task of
checking un-convergence to a single quasi-interior point, if one
exists; cf \cite[Lemma 24]{Troitsky:04}.

\begin{lemma}\label{un-qip}
Let $X$ be a Banach lattice with a quasi-interior point $e$. Then $x_\alpha \goesun 0$ iff $\abs{x_\alpha}\wedge e\to 0$.
\end{lemma}

\begin{proof}									
The forward implication is immediate. For the reverse implication, let $u\in X_+$ be arbitrary and fix $\varepsilon>0$. Note that
\begin{displaymath}
  \abs{x_\alpha}\wedge u
  \le\abs{x_\alpha}\wedge(u-u\wedge me)+\abs{x_\alpha}\wedge(u\wedge me)
  \le(u-u\wedge me)+m\bigl(\abs{x_\alpha}\wedge e\bigr)
\end{displaymath}
and, therefore,
\begin{displaymath}
  \bignorm{\abs{x_\alpha}\wedge u}
  \le\norm{u-u\wedge me}+m\bignorm{\abs{x_\alpha}\wedge e}
\end{displaymath}
for all $\alpha$ and all $m\in\mathbb N$. Since $e$ is quasi-interior, we can find $m$ such that $\norm{u-u\wedge me}<\varepsilon$. Furthermore, it follows from $\abs{x_\alpha}\wedge e\to 0$ that there exists $\alpha_0$ such that $\bignorm{\abs{x_\alpha}\wedge e}<\frac{\varepsilon}{m}$ whenever $\alpha\ge\alpha_0$. It follows that
\begin{math}
  \bignorm{\abs{x_\alpha}\wedge u}<\varepsilon+m\frac{\varepsilon}{m}=2\varepsilon.
\end{math}
Therefore, $\abs{x_\alpha}\wedge u\to 0$.
\end{proof}

\begin{corollary} \label{un-wu}
Let $X$ be an order continuous Banach lattice with a weak unit $e$. Then $x_\alpha \goesun 0$ iff $\abs{x_\alpha}\wedge e\to 0$. 
\end{corollary}

\begin{proof}
If $X$ is order continuous, then $e$ is a weak unit iff $e$ is a quasi-interior point.
\end{proof}

Recall that norm convergence is sequential in nature. In particular, given a net $(x_\alpha)$ in a normed space, if $x_\alpha\to x$ then there exists an increasing sequence of indices $(\alpha_n)$ such that $x_{\alpha_n}\to x$. This often allows one to reduce nets to sequences when dealing with norm convergence. In view of \Cref{un-qip}, we can do the same with the un-convergence as long as the space has a quasi-interior point (in particular, when $X$ is separable):

\begin{corollary}\label{qui-seq}
  Suppose that $X$ has a quasi-interior point and $x_\alpha\goesun 0$ for some net $(x_\alpha)$ in $X$. Then there exists an increasing sequence of indices $(\alpha_n)$ such that $x_{\alpha_n}\goesun 0$.
\end{corollary}

\begin{question}\label{un-seq}
  Does \Cref{qui-seq} remain valid without a quasi-interior point?
\end{question}

We will show in \Cref{un-oc-seq} that the answer is affirmative for order continuous spaces.

\section{Disjoint subsequences}


The following lemma is standard; we provide the proof for the convenience of the reader.

\begin{lemma}\label{GRDP}
  Let $\abs{x}=u+v$ for some vector $x$ and some positive vectors $u$ and $v$ in a vector lattice. Then there exist $y$ and $z$ such that $x=y+z$, $\abs{y}=u$, and $\abs{z}=v$.
\end{lemma}

\begin{proof}
  Applying the Riesz Decomposition Property \cite[Theorem~1.20]{Aliprantis:06} to the equality $x^++x^-=u+v$, we find four positive vectors vectors $a$, $b$, $c$, and $d$ such that $u=a+b$, $v=c+d$, $x^+=a+c$, and $x^-=b+d$. Put $y=a-b$ and $z=c-d$. Then $y+z=x^+-x^-=x$.  It follows from $0\le a\le x^+$ and $0\le b\le x^-$ that $a\perp b$ and, therefore, $\abs{y}=\abs{a-b}=a+b=u$. Similarly, $c\perp d$, and, therefore, $\abs{z}=v$.
\end{proof}

\begin{theorem}\label{KP}
  Let $(x_\alpha)$ be a net in $X$ such that $x_\alpha\goesun 0$. Then there exists an increasing sequence of indices $(\alpha_k)$ and a disjoint sequence $(d_k)$ such that $x_{\alpha_k}-d_k\to 0$.
\end{theorem}

\begin{proof}
  Assume first that $x_\alpha\ge 0$ for every $\alpha$. Pick any $\alpha_1$. Suppose that $\alpha_1,\dots,\alpha_{k-1}$ have been constructed. Note that $x_\alpha\wedge x_{\alpha_i}\to 0$ for every $i=1,\dots,k-1$. Choose $\alpha_k>\alpha_{k-1}$ so that 
  \begin{math}
    \bignorm{x_{\alpha_k}\wedge x_{\alpha_i}}\le\frac{1}{2^{k+i}}
  \end{math}
  for every $i=1,\dots,k-1$. This produces an increasing sequence of indices $(\alpha_k)$ such that $\norm{z_{ik}}\le\frac{1}{2^{k+i}}$ where $z_{ik}=x_{\alpha_i}\wedge x_{\alpha_k}$, $1\le i<k$.

For every $k$, put $v_k=\sum_{i=1}^{k-1}z_{ik}+\sum_{j=k+1}^\infty z_{kj}$. Clearly, $v_k$ is defined and $\norm{v_k}<\frac{1}{2^k}$.
Put $d_k=(x_{\alpha_k}-v_k)^+$. It is easy to see that $0\le x_{\alpha_k}-d_k\le v_k$, so that $\norm{x_{\alpha_k}-d_k}\to 0$ as $k\to\infty$. It is left to show that the sequence $(d_k)$ is disjoint. Let $k<m$. Then
  \begin{eqnarray*}
    d_k&=&(x_{\alpha_k}-v_k)^+\le(x_{\alpha_k}-z_{km})^+
      =x_{\alpha_k}-x_{\alpha_k}\wedge x_{\alpha_m},\mbox{ and }\\
    d_m&=&(x_{\alpha_m}-v_m)^+\le(x_{\alpha_m}-z_{km})^+
      =x_{\alpha_m}-x_{\alpha_k}\wedge x_{\alpha_m}.
  \end{eqnarray*}
  It follows that $d_k\perp d_m$.

  For the general case, we first apply the first part of the proof to
  the net $\bigl(\abs{x_\alpha}\bigr)$ and produce an increasing
  sequence of indices $(\alpha_k)$ and two positive sequences $(w_k)$
  and $(h_k)$ such that $\abs{x_{\alpha_k}}=w_k+h_k$, $(w_k)$ is
  disjoint, and $h_k\to 0$. By \Cref{GRDP}, we can find sequences
  $(d_k)$ and $(g_k)$ in $X$ with $\abs{d_k} = w_k$, $\abs{g_k} = h_k$
  and $x_{\alpha_k} = d_k + g_k$. It follows that $(d_k)$ is a
  disjoint sequence and $g_k\to 0$. Thus,
  $x_{\alpha_k}-d_k\to 0$.
\end{proof}

\begin{remark}
  \Cref{KP} is a variant of the Kade\v c-Pe{\l}\-czy{\'n}\-ski dichotomy theorem; cf \cite[p.38]{Lindenstrauss:79}. \Cref{KP} clearly implies \cite[Lemma~6.7]{GTX}; unlike in \cite[Lemma~6.7]{GTX}, we do not require the space to be order continuous or the net be norm bounded. Also, we start with a net instead of a sequence. 
\end{remark}

Recall the following standard fact; see, e.g., Exercise~13 in \cite[p.~25]{Abramovich:02}.

\begin{proposition}\label{norm-o-subseq}
  Every norm convergent sequence in a Banach lattice has a subsequence which converges in order to the same limit. 
\end{proposition}

\begin{corollary}\label{un-oc-seq}
  Let $(x_\alpha)$ be a net in an order continuous Banach lattice $X$ such that $x_\alpha\goesun 0$. Then there exists an increasing sequence of indices $(\alpha_k)$ a such that $x_{\alpha_k}\goesuo 0$ and $x_{\alpha_k}\goesun 0$.
\end{corollary}

\begin{proof}
  Let $(\alpha_k)$ and $(d_k)$ be as in \Cref{KP}. Since $(d_k)$ is disjoint, we have $d_k\goesuo 0$ and, therefore, $d_k\goesun 0$. It now follows from $x_{\alpha_k}-d_k\to 0$ that $x_{\alpha_k}-d_k\goesun 0$ and, therefore,  $x_{\alpha_k}\goesun 0$. Furthermore, since $x_{\alpha_k}-d_k\to 0$, passing to a further subsequence, we may assume that $x_{\alpha_k}-d_k\goeso 0$ and, therefore, $x_{\alpha_k}-d_k\goesuo 0$. This yields $x_{\alpha_k}\goesuo 0$.
\end{proof}

Note that \Cref{un-oc-seq} provides a partial answer to \Cref{un-seq}.

\section{Uo-convergent subsequences and convergence in measure}

We now have an analogue of \Cref{norm-o-subseq} for un- and uo-convergences. 

\begin{proposition} \label{uo-subseq}
If $x_n \goesun 0$ then there is a subsequence $(x_{n_k})$ of $(x_n)$ such that $x_{n_k} \goesuo 0$.
\end{proposition}   

\begin{proof}
Define $e := \sum_{n=1}^{\infty} \frac{\abs{x_n}}{2^n \norm{x_n}}$. Let $B_e$ be the band generated by $e$ in $X$. It follows from $x_n\goesun 0$ that $\abs{x_n}\wedge e\to 0$ in $X$ and, therefore, in $B_e$. There exists a subsequence $(x_{n_k})$ of $(x_n)$ such that $\abs{x_{n_k}}\wedge e\goeso 0$ in $B_e$. Since $e$ is a weak unit in $B_e$, we have $x_{n_k}\goesuo 0$ in $B_e$. Finally, since $B_e$ is an ideal in $X$, it follows from \cite[Corollary~3.8]{GTX} that 
$x_{n_k} \goesuo 0$ in $X$. 
\end{proof}

It was observed in \cite[Example~23]{Troitsky:04} that for sequences in $L_p(\mu)$, where $\mu$ is a finite measure, un-convergence agrees with convergence in measure. We now provide an alternative proof of this fact.

\begin{corollary}(\cite{Troitsky:04}) \label{un-Lp}
Let $(f_n)$ be a sequence in $L_{p}(\mu)$ where $1\le p<\infty$ and $\mu$ is a finite measure. Then $f_n \goesun 0$ iff $f_n \goesmu 0$.
\end{corollary}

\begin{proof}
  Without loss of generality, $f_n\ge 0$ for all $n$. Suppose $f_n\goesmu 0$. It is easy to see that $f_n\wedge\one\to 0$ in the norm of $L_p(\mu)$. It follows from \Cref{un-qip} that $f_n\goesun 0$.

Conversely, suppose that $f_n \goesun 0$. Then every subsequence $(f_{n_k})$ is still un-null and, therefore, by \Cref{uo-subseq}, has a further subsequence $f_{n_{k_i}}$ such that $f_{n_{k_i}} \goesuo 0$ and, therefore, $f_{n_{k_i}} \goesae 0$. This yields $f_n \goesmu 0$.  
\end{proof}

\begin{remark}\label{ae-mu}
  In the last step of the preceding proof, we used the fact that given a sequence of measurable functions over a measure space with a finite measure, the sequence converges in measure iff every subsequence has a further subsequence which converges a.e.\ (to the same limit); see, e.g., Exercise~22 in \cite[p.~96]{Royden:88}. Note that \Cref{uo-subseq} may be viewed as an extension of one of the directions of this equivalence to general Banach lattices. The next result shows that for order continuous Banach lattices the other direction extends as well.
\end{remark}

\begin{theorem} \label{un-subseq-uo}
  A sequence in an order continuous Banach lattice $X$ is un-null iff every subsequence has a further subsequence which uo-converges to zero.
\end{theorem}

\begin{proof}
  The forward implication is \Cref{uo-subseq}. To show the converse, assume that $x_n\not\goesun 0$. Then there exist $\delta>0$, $u \in X_+$, and a subsequence $(x_{n_k})$ such that $\bignorm{\abs{x_{n_k}}\wedge u}>\delta$ for all $k$. By assumption, there is a subsequence $(x_{n_{k_i}})$ of $(x_{n_k})$ such that $x_{n_{k_i}} \goesuo 0$, and, therefore, $x_{n_{k_i}} \goesun 0$ by  \Cref{uo-oc-un}. This yields $\abs{x_{n_{k_i}}}\wedge u\to 0$, which is a contradiction. 
\end{proof}

\begin{remark}									  Again, Example~\ref{disj-ellinfty} shows that the order continuity assumption cannot be removed.
\end{remark} 

Suppose that $X$ is an order continuous Banach lattice with a weak unit $e$. It is known that $X$ can be represented as an order and norm dense ideal in $L_1(\mu)$ for some finite measure $\mu$. That is, there is a vector lattice isomorphism $T\colon X\to L_1(\mu)$ such that $\Range T$ is an order and norm dense ideal of $L_1(\mu)$. Note that $T$ need not be a norm isomorphism, though $T$ may be chosen to be continuous and $Te=\one$. Moreover, $\Range T$ contains $L_\infty(\mu)$ as a norm and order dense ideal. It is common to identify $X$ with $\Range T$ and just view $X$ as an ideal of $L_1(\mu)$; we also identify $e$ with $\one$. We call such an inclusion of $X$ into an $L_1(\mu)$ space an \term{AL-representation} of $X$. We refer the reader to \cite[Theorem~1.b.14]{Lindenstrauss:79} or \cite[Section~4]{GTX} for more details on AL-representations.

It was observed in \cite[Remark~4.6]{GTX} that for a sequence $(x_n)$ in $X$, $x_n\goesuo 0$ in $X$ iff $x_n\goesae 0$ in $L_1(\mu)$. We prove an analogous result for un-convergence.

\begin{theorem}\label{AL-un}
  Let $X$ be an order continuous Banach lattice with a weak unit; let $L_1(\mu)$ be an AL-representation for $X$. For a sequence $(x_n)$ in $X$, we have $x_n \goesun 0$ in $X$ iff $x_n \goesmu 0$ in $L_{1}(\mu)$.
\end{theorem}

\begin{proof}
  By \Cref{un-subseq-uo}, $x_n\goesun 0$ in $X$ iff for every subsequence $(x_{n_k})$ there is a further subsequence $(x_{n_{k_i}})$ such that $x_{n_{k_i}}\goesuo 0$. The latter is equivalent to $x_{n_{k_i}}\goesae 0$. Now apply \Cref{ae-mu}.
\end{proof}

\section{When do un- and uo-convergences agree?}

Our next goal is to prove that uo- and un-convergences for sequences agree iff $X$ is order continuous and atomic. Recall that a non-zero element $a\in X_+$ is an \term{atom} iff the ideal $I_a$ consists only of the scalar multiples of $a$. In this case, $I_a$ is a projection band. Let $P_a$ be the corresponding band projection. We say that $X$ is \term{atomic} or \term{discrete} if it equals the band generated by all the atoms in it. Suppose that $X$ is atomic, and fix a maximal disjoint collection $A$ of atoms in $X$. For every $x\in X_+$, we have $x=\bigvee_{a\in A}x_aa$, where $x_aa=P_ax$. One can also write this as $x=\sum_{a\in A}x_aa$, where the sum is understood as the order limit (or the supremum) of sums over finite subsets of $A$. This sum may be viewed as a coordinate expansion of $x$ over $A$. Furthermore, if we are also given a positive vector $y=\sum_{a\in A}y_aa$, then $x\le y$ iff $x_a\le y_a$ for every $a\in A$. Suppose now that, in addition, $X$ is order continuous. Then it can be shown that only countably many of coefficients $x_a$ are non-zero. Enumerating them, we get $x=\sum_{i=1}^\infty x_{a_i}a_i$, where the series converges in order and, therefore, in norm. For details, see \cite{Aliprantis:03}, Exercise~7 in \cite[p.~143]{Schaefer:74}, and the proof of Proposition~1.a.9 in~\cite{Lindenstrauss:79}. We will need two standard lemmas.

\begin{lemma}\label{atom-oconv}
  Suppose that $X$ is atomic and order continuous, and $(x_n)$ is an order bounded sequence in $X$. If $x_n\to 0$ then $x_n\goeso 0$. 
\end{lemma}

\begin{proof}
  Without loss of generality, $x_n\ge 0$ for all $n$. Let $u\in X_+$ such that $x_n\le u$ for every $n$. There is a sequence of distinct atoms $(a_i)$ in $A$ such that $u=\sum_{i=1}^\infty u_ia_i$ for some coefficients $(u_i)$; the series converges in norm and in order. Given $n\in\mathbb N$, it follows from $0\le x_n\le u$ that we can write $x_n=\sum_{i=1}^\infty x_{ni}a_i$ for some sequence of coefficients $(x_{ni})$. Note that $0\le x_{ni}a_i\le x_n$ for every $n$ and $i$; it follows that $\lim_nx_{ni}=0$ for every $i$; that is, the sequence $(x_n)$ converges to zero ``coordinate-wise''. 

For each $k\in\mathbb N$, define 
\begin{displaymath}
  v_k=\sum_{i=1}^k(\tfrac1k\wedge u_i)a_i+\sum_{i=k+1}^\infty u_ia_i.
\end{displaymath}
It is easy to see that $v_k\downarrow 0$. On the other hand, since $x_n\le u$ and $(x_n)$ converges to zero coordinate-wise, for every $k$ we can find $n_k$ such that $x_n\le v_k$ whenever $n\ge n_k$. It follows that $x_n\goeso 0$.
\end{proof}

\begin{lemma}\label{mu-nonatom}
  If $\mu$ is a finite non-atomic measure then there exists a sequence $(f_n)$ in $L_\infty(\mu)$ which converges to zero in measure but not a.e..
\end{lemma}

\begin{proof}
  In the special case when $\mu$ is the Lebesgue measure on the unit interval, we take $(f_n)$ to be the ``typewriter'' sequence 
  \begin{displaymath}
    f_n=\chi_{[\frac{n-2^k}{2^k}, \frac{n-2^k +1}{2^k}]}
    \text{ where $k\ge 0$ such that }
    2^k\le n<2^{k+1}.
  \end{displaymath}
 In the general case, we produce a similar sequence using Exercise~2 in \cite[p.~174]{Halmos:70}.
\end{proof}

\begin{theorem}\label{un-iff-uo}
The following are equivalent:
\begin{enumerate}
\item\label{un-iff-uo-iff} $x_n \goesuo 0 \iff x_n \goesun 0$ for every sequence $(x_n)$ in $X$;
\item\label{un-iff-uo-atom} $X$ is order continuous and atomic. 
\end{enumerate}
\end{theorem}

\begin{proof}
  \eqref{un-iff-uo-atom}$\Rightarrow$\eqref{un-iff-uo-iff} Suppose $X$ is order continuous and atomic. The implication $x_n \goesuo 0$ $\Rightarrow$ $x_n \goesun 0$ is trivial, and the reverse implication follows immediately from \Cref{atom-oconv}.

\eqref{un-iff-uo-iff}$\Rightarrow$ \eqref{un-iff-uo-atom} Let $(x_n)$ be a disjoint order bounded sequence in $X$. Then $x_n\goesuo 0$ by \cite[Corollary~3.6]{GTX}. By assumption, $x_n\goesun 0$. Since the sequence is order bounded, this yields $x_n\to 0$. Hence, $X$ is order continuous. It follows that every closed ideal in $X$ is a projection band.

It remains to show that $X$ is atomic. Suppose not; then the band $X_1$ generated by all the atoms in $X$ is a proper subset of $X$. Let $X_2$ be the complementary band. Fix a non-zero $w\in X_2^+$, and let $Y=B_w$, the band generated by $w$ in $X$. Clearly, $Y\subseteq X_2$, $Y$ is order continuous, $w$ is a weak unit in $Y$, and $Y$ has no atoms. We can find an AL-representation for $Y$ such that $L_\infty(\mu)\subseteq Y\subseteq L_1(\mu)$. Since $Y$ has no atoms, it is easy to see that $\mu$ is a non-atomic measure.

By \Cref{mu-nonatom}, there is a sequence $(x_n)$ in $L_\infty(\mu)$ such that $x_n\goesmu 0$ but $x_n\not\goesae 0$. It follows that $x_n\goesun 0$ but $x_n\not\goesuo 0$ in $Y$ and, therefore, in $X$; a contradiction.
\end{proof}

\section{Un-convergence and weak convergence}

In this section, we consider the relationship between un- and weak convergences. For a mononote net, weak convergence implies norm convergence, and, therefore, un-convergence; however, weak convergence does not imply un-convergence in general. For example, the following fact was observed in \cite[Theorem~2.2]{Chen:98}:

\begin{lemma}(\cite{Chen:98})\label{Chen}
   If $X$ is non-atomic and order continuous, and $x\in X_+$, then there exists a sequence $(x_n)$ such that $x_n\goesw 0$ yet $\abs{x_n}=x$ for all $n$.
\end{lemma}

Clearly, $(x_n)$ is not un-null.

\begin{proposition} \label{w-un}
  The following are equivalent:
  \begin{enumerate}
    \item\label{w-un-to} $x_n \goesw 0$ implies $x_n \goesun 0$ for every sequence $(x_n)$ in $X$;
  \item\label{w-un-atom} $X$ is order continuous and atomic.
  \end{enumerate}
\end{proposition}

\begin{proof}
  \eqref{w-un-to}$\Rightarrow$\eqref{w-un-atom}
The proof is similar to that of \Cref{un-iff-uo}. Let $(x_n)$ be a disjoint order bounded sequence in $X$. Then $x_n \goesw 0$. By assumption, $x_n\goesun 0$. Since $(x_n)$ is order bounded, this yields $x_n\to 0$. Therefore, $X$ is order continuous.

  Now suppose that $X$ is not atomic. Then $X=X_1\oplus X_2$ where $X_1$ is the band generated by the atoms, and $X_2$ is the complementary band. Since $X$ is not atomic, we have $X_2\ne\{0\}$. Now apply \Cref{Chen} to $X_2$.

  \eqref{w-un-atom}$\Rightarrow$\eqref{w-un-to} Lemma~6.14 in \cite{GTX} asserts that if $X$ is atomic and $x_n\goesw  0$ then $x_n\goesuo 0$. If, in addition, $X$ is order continuous, this yields $x_n\goesun 0$.
\end{proof}

In the previous result, the conditions for a weakly-null sequence to be un-null are quite strong. On the other hand, in an arbitrary Banach lattice, this fact is true for monotone nets. If we remove the constraint that $X$ is atomic, then we obtain the following result. Recall that for an order bounded positive net $(x_\alpha)$ in an order continuous Banach lattice, $x_\alpha\goesw 0$ implies $x_\alpha\to 0$ by \cite[Theorem 4.17]{Aliprantis:06}.

\begin{proposition}
  Let $(x_\alpha)$ be a positive net in an order continuous Banach lattice $X$. If $x_\alpha\goesw 0$ then $x_\alpha\goesun 0$.
\end{proposition}

\begin{proof}
  For every $u\in X_+$ we have $0\le x_\alpha\wedge u\le x_\alpha$ for every $\alpha$. This yields $x_\alpha\wedge u\goesw 0$. Since this net is positive and order bounded, $x_\alpha\wedge u\to 0$. Hence, $x_\alpha\goesun 0$.
\end{proof}

We have now seen several conditions on $X$ that yield $x_n \goesw 0$ implies $x_n \goesun 0$. For the converse, we have the following.  

\begin{theorem}
   If $X^*$ is order continuous then $x_\alpha\goesun 0$ implies $x_\alpha\goesw 0$ for every norm bounded net $(x_\alpha)$ in $X_+$.
\end{theorem}

\begin{proof}
  Suppose that $X^*$ is order continuous and $(x_\alpha)$ is a norm bounded net in $X_+$ with $x_\alpha\goesun 0$. Without loss of generality, $\norm{x_\alpha}\le 1$ for every $\alpha$. Since $X^*$ is order continuous, it follows from \cite[Theorem 4.19]{Aliprantis:06} that for every $\varepsilon>0$ and every $f\in X^*_+$ there exists $u\in X_+$ such that
  \begin{math}
    f\bigl(\abs{x}-\abs{x}\wedge u\bigr)<\varepsilon
  \end{math}
whenever $\norm{x}\le 1$. In particular,
  \begin{math}
    f\bigl(x_\alpha-x_\alpha\wedge u\bigr)<\varepsilon
  \end{math}
for every $\alpha$. It follows from $x_\alpha\wedge u\to 0$ that $f(x_\alpha)<\varepsilon$ for all sufficiently large $\alpha$. Hence, $f(x_\alpha)\to 0$ and, therefore, $x_\alpha\goesw 0$.
\end{proof}

We do not know whether the converse is true. If $X^*$ is not order continuous then $\ell_1$ is lattice embeddable in $X$; see, e.g., \cite[Theorem 4.69]{Aliprantis:06}. Let $(e_n)$ be the standard basis of $\ell_1$ viewed as a sequence in $X$. It is easy to see that $e_n \not \goesw 0$ in $\ell_1$ and, therefore, in $X$. It is also easy to see that $e_n \goesun 0$ in $\ell_1$. However, this does not imply that $e_n \goesun 0$ in $X$; see \Cref{un-sublat-not}.

\section{Un-convergence is topological}

It is well known that a.e.\ convergence is not topological; see, e.g.,
\cite{Ordman:66}. That is, this convergence is not given by a
topology. It follows that uo-convergence need not be topological in
general. We show that un-convergence is topological. Moreover, we
explicitly define the neighborhoods of this topology.

Given an $\varepsilon>0$ and a non-zero $u\in X_+$, we put
\begin{displaymath}
  V_{u,\varepsilon}=\bigl\{x\in X\mid \bignorm{\abs{x}\wedge u}<\varepsilon\bigr\}.
\end{displaymath}
Let $\mathcal N_0$ be the collection of all the sets of this form. We claim that $\mathcal N_0$ is a base of neighborhoods of zero for some Hausdorff linear topology. Once we establish this, we will be able to define arbitrary neighborhoods as follows: a subset $U$ of $X$ is a neighbourhood of $y$ if $y+V\subseteq U$ for some $V\in\mathcal N_0$. It follows immediately from the definition of un-convergence that $x_\alpha\goesun 0$ iff every set in $\mathcal N_0$ contains a tail of this net, hence the un-convergence is exactly the convergence given by this topology.

We have to verify that $\mathcal N_0$ is indeed a base of neighborhoods of zero; cf. \cite[Theorem~5.1]{Kelley:76} or \cite[Theorem 3.1.10]{Runde:05}.

First, every set in $\mathcal N_0$ trivially contains zero.

Second, we need to show that the intersection of any two sets in $\mathcal N_0$ contains another set in $\mathcal N_0$. Take $V_{u_1,\varepsilon_1}$ and $V_{u_2,\varepsilon_2}$ in $\mathcal N_0$. Put $\varepsilon=\varepsilon_1\wedge\varepsilon_2$ and $u=u_1\vee u_2$. We claim that 
\begin{math}
  V_{u,\varepsilon}\subseteq V_{u_1,\varepsilon_1}\cap V_{u_2,\varepsilon_2}.
\end{math}
Indeed, take any $x\in V_{u,\varepsilon}$. Then $\bignorm{\abs{x}\wedge u}<\varepsilon$.
It follows from $\abs{x}\wedge u_1\le\abs{x}\wedge u$ that
\begin{displaymath}
  \bignorm{\abs{x}\wedge u_1}\le\bignorm{\abs{x}\wedge u}
  <\varepsilon\le\varepsilon_1,
\end{displaymath}
so that $x\in V_{u_1,\varepsilon_1}$. Similarly,  $x\in V_{u_2,\varepsilon_2}$.

It is easy to see that $V_{u,\varepsilon}+V_{u,\varepsilon}\subseteq V_{u,2\varepsilon}$. 
This immediately implies that for every $U$ in $\mathcal N_0$ there exists $V\in\mathcal N_0$ such that $V+V\subseteq U$. It is also easy to see that for every $U\in\mathcal N_0$ and every scalar $\lambda$ with $\abs{\lambda}\le 1$ we have $\lambda U\subseteq U$. 

Next, we need to show that for every $U\in\mathcal N_0$ and every $y\in U$, there exists $V\in\mathcal N_0$ such that $y+V\subseteq U$. Let $y\in V_{u,\varepsilon}$ for some $\varepsilon>0$ and a non-zero $u\in X_+$. We need to find $\delta>0$ and a non-zero $v\in X_+$ such that $y+V_{v,\delta}\subseteq V_{u,\varepsilon}$. Put $v:=u$. It follows from $y\in V_{u,\varepsilon}$ that $\bignorm{\abs{y}\wedge u}<\varepsilon$; take
\begin{math}
  \delta:=\varepsilon-\bignorm{\abs{y}\wedge u}.
\end{math}
We claim that  $y+V_{v,\delta}\subseteq V_{u,\varepsilon}$. Let $x\in V_{v,\delta}$; it suffices show that $y+x\in V_{u,\varepsilon}$. Indeed,
\begin{math}
  \abs{y+x}\wedge u\le\abs{y}\wedge u+\abs{x}\wedge u,
\end{math}
so that
\begin{displaymath}
  \bignorm{\abs{y+x}\wedge u}\le
  \bignorm{\abs{y}\wedge u}+\bignorm{\abs{x}\wedge u}
  <\bignorm{\abs{y}\wedge u}+\delta=\varepsilon.
\end{displaymath}

Finally, in order to show that the topology is Hausdorff, we need to verify that $\bigcap\mathcal N_0=\{0\}$. Indeed, suppose that $0\ne x\in V_{u,\varepsilon}$ for all non-zero $u\in X_+$ and $\varepsilon>0$. In particular, $x\in V_{\abs{x},\varepsilon}$, so that $\norm{x}=\bignorm{\abs{x}\wedge\abs{x}}<\varepsilon$ for every $\varepsilon>0$; a contradiction.

Note that we could also conclude that the topology is linear and Hausdorff from \Cref{basic}.

\textbf{Acknowledgement.} We would like to thank Niushan Gao for valuable discussions.

\end{document}